\pdfoutput=1
\documentclass[10pt]{article}
\usepackage[a4paper, margin=1in]{geometry}
\usepackage{changepage}
\usepackage{amsmath, amssymb, amsthm}
\usepackage[utf8]{inputenc}
\usepackage[english]{babel}
\usepackage[numbers]{natbib}
\usepackage{url}
\usepackage[usenames]{color}
\usepackage[colorlinks=true]{hyperref}
\usepackage{cleveref}
\usepackage{tabularx}
\usepackage{breqn}
\usepackage[title]{appendix}
\usepackage{xcolor}
\usepackage{textcomp}
\theoremstyle{plain}

\newtheorem{conjecture}{Conjecture}[section]
\newtheorem{corollary}{Corollary}[section]
\newtheorem{lemma}{Lemma}[section]

\newtheorem{theorem}{Theorem}[section]
\newtheorem*{theorem*}{Theorem}
\newtheorem{remark}{Remark}[section]

\crefname{conjecture}{Conjecture}{Conjectures}
\crefname{theorem}{Theorem}{Theorems}
\crefname{theorem*}{Theorem}{Theorems}
\crefname{corollary}{Corollary}{Corollaries}
\crefname{lemma}{Lemma}{Lemmas}
\crefname{proposition}{Proposition}{Propositions}
\crefname{remark}{Remark}{Remarks}
\crefname{note}{Note}{Notes}
\crefname{definition}{Definition}{Definitions}
\crefname{notation}{Notation}{Notations}
\crefname{example}{Example}{Examples}
\crefname{question}{Question}{Questions}
\crefname{section}{\S}{Sections}
\crefname{equation}{Equation}{Equations}
\crefformat{equation}{(eq.~#2#1#3)}

\newcommand{\floor}[1]{\left\lfloor #1 \right\rfloor}

\newcommand{\CTerm}{\mathcal{C}}

\newcommand{\HW}{\operatorname{HW}}

\newcommand{\Z}{\mathbb{Z}}

\title{Elementary closed-forms for non-trivial divisors}
\author{Mihai Prunescu \footnote{Research Center for Logic, Optimization and Security (LOS), Faculty of Mathematics and Computer Science, University of Bucharest, Academiei 14, Bucharest (RO-010014), Romania; Institute of Logic and Data Science, Bucharest, Romania; Simion Stoilow Institute of Mathematics of the Romanian Academy, Research unit 5, P. O. Box 1-764, Bucharest (RO-014700), Romania. E-mail: {\tt mihai.prunescu@imar.ro}, {\tt mihai.prunescu@gmail.com}.},
Joseph M. Shunia \footnote{Independent researcher, Ann Arbor, Michigan, United States, e-mail address: {\tt jshunia@gmail.com}.}
}
\date{October 2025}

\begin{document}

\maketitle

\begin{abstract} \noindent
We present several elementary closed-forms that express a non-trivial divisor for every composite integer $n > 1$. Each closed-form consists of a fixed number of elementary arithmetic operations drawn from the set:
addition, subtraction, multiplication, integer division, and exponentiation.
\\[2mm]
\noindent Two families of closed-forms are developed. First, direct application of the hypercube method yields closed-forms $T_1(n)$, $T_2(n)$, $T_3(n)$, and $T_4(n)$ expressing the smallest prime divisor,
largest non-trivial divisor, largest prime divisor, and greatest prime $\leq n$, respectively. The factorial-unwinding technique underlying these hypercube constructions leads to extreme symbolic complexity, motivating our main result: An alternative closed-form $T(n)$ that avoids factorial-unwinding by synthesizing the quadratic residue invariants $\chi(n)$ (largest $r$ such that $r^2$ is a divisor) and $\omega(n)$ (number of distinct prime divisors) with integer root extraction.
\\[2mm]
\noindent Although evaluating these closed-forms requires exponential time, the number of arithmetic operations performed remains constant and independent of the input size $n$. This sharply contrasts with traditional algorithmic methods, where the number of operations required to locate a non-trivial divisor necessarily scales with $n$.
\\[2mm]
\end{abstract}

\section{Introduction}
The \textbf{integer factorization problem} is the computational task of decomposing a given natural number $n$ into its prime factors. Formally, given an integer $n > 1$, the objective is to find a set of primes $\{ p_1, p_2, \dots, p_k \}$ such that $n = p_1^{e_1} \cdot p_2^{e_2} \cdots p_k^{e_k}$, where $e_i \geq 1$ are integer exponents. This problem is fundamental in computational number theory and has broad implications in fields such as cryptography, where the security of widely used protocols like RSA is based on the difficulty of factorizing large \textbf{semiprimes} \cite{rivest1978rsa}, natural numbers $n=pq$ that are the product of two prime numbers.

The problem is often presented in the following way: \textit{Given a natural number $n$, find a proper divisor $d$ with $1 < d < n$, respectively decide that such a divisor does not exist.}

If such an algorithm is applied for the numbers $d$ and $n/d$, and the procedure is repeated  as far as possible, one obtains a binary tree of results. Its terminal nodes are then the prime divisors of $n$, with the given multiplicity, and the factorization is done.

At this point it is worth to remark that a fast decision (polynomial time) about the primality of $n$ is possible by the AKS algorithm \cite{AKS, AKSErrata, granvilleprimes}, but we don't know how to efficiently find a non-trivial divisor of $n$, when $n$ is not a prime number, by a fast deterministic algorithm. On the other hand, Shor's quantum factorization algorithm, see \cite{shor}, promises to find a proper divisor $d$ by constructing and running a quantum circuit of polynomial complexity.

\textbf{Arithmetic terms} are defined in \cite{prunescu2024factorial, prunescu2024numbertheoreticfunctions} as functions  constructed in the language
\begin{align*}
L = \{ +, \dot{-}, \ast, /, \textup{\textasciicircum} \} ,
\end{align*}
by the usual construction rules of term construction in the first-order logic. The constants $0$, $1$ and the variables are arithmetic terms. Further, if $t_1$ and $t_2$ are arithmetic terms, then the following expressions are arithmetic terms: $(t_1 + t_2)$, $(t_1 \dot{-} t_2)$, $(t_1 \ast t_2)$, $(t_1 / t_2)$ and $2^{t_1}$ are arithmetic terms.
Here ($/$) denotes the quotient obtained by the integer division and the \textbf{monus} operator ($\dot{-}$) denotes bounded subtraction: $t_1 \dot{-} t_2 = \max(t_1 - t_2, 0)$. In this article, the notation $\dot{-}$ will be used only if it is important to distinguish this operation from the normal subtraction.

Put simply: Arithmetic terms are \textbf{fixed-length elementary closed-form expressions} and might informally be referred to as \textbf{formulas}.

We present the first arithmetic terms that return a non-trivial divisor for every composite integer $n > 1$. Previous work has explored arithmetic terms for factoring semiprimes. For instance, in \cite{shunia2024elementary} Shunia presented an arithmetic for factoring semiprimes by direct application of a term for $\floor{\sqrt{n}}$. A semiprime factoring term was also done independently by Prunescu and Sauras-Altuzarra in \cite{prunescu2024numbertheoreticfunctions}, who instead utilized a term for Euler's totient function $\varphi(n)$ to obtain factors. It could be said that our general purpose factoring terms represent the logical conclusion of these earlier semiprime-only factoring terms.

\begin{remark}
Our results also extend prior work by Shamir \cite{shamir1979factoring} (the same Shamir who contributed the ``S'' in RSA \cite{rivest1978rsa}), who gave an algorithmic method for factoring integers using $O_{\textup{ops}}(\log n)$ arithmetic operations, where each exponentiation $n^m$ is counted as a single operation. Under this same accounting, our closed-form expressions represent only $O_{\textup{ops}}(1)$ operations, and thus achieve the minimal possible count. However, this does not imply practical efficiency: The symbolic evaluation of our expressions incurs exponential time complexity $O(2^n)$ due to subterms whose bit complexity grows exponentially in the input size.
\end{remark}

\subsection{Structure of the paper}
\begin{enumerate}
    \item In \cref{sectionpreliminaries} we recall the hypercube method. Given an exponential Diophantine equation in unknowns $(x_1, \dots,x_k)$, say $E(\vec n, \vec x) = 0$, a box $[0,t]^k$ and a function $u(\vec n, t)$ such that for all $\vec x \in [0,t]^k$ and all $\vec n \in N^m$ one has:
    $$0 \leq E(\vec n, \vec x) < 2^{u(\vec n, t)},$$
    it produces an arithmetic term $t(\vec n, t, u(\vec n, t))$ that counts the number of solutions $\vec x \in [0,t]^k$ for the exponential Diophantine equation $E(\vec n, \vec x) = 0$. Also  an arithmetic term expressing the function $n!$ is recalled. Moreover, the hypercube method needs arithmetic terms expressing the function $\nu_2(n)$, meaning the value of the $2$-ary valuation in $n$, respectively $\gcd(m,n)$, the greatest common divisor of the natural numbers $m$ and $n$. These arithmetic terms are recalled as well. They will play a role also by themselves, independently of the hypercube method.
    \item In \cref{sectiondirectapplications} we show how direct applications of the hypercube method yield valid arithmetic terms for factoring: $T_1(n)$ computing the least prime divisor, $T_2(n)$ the greatest non-trivial divisor, $T_3(n)$ the greatest prime divisor, and $T_4(n)$ the greatest prime less than or equal to $n$. As these constructions are based on unwinding an arithmetic term expressing the factorial function $n!$ via Wilson's theorem (writing the relation $a = b!$ as a positive exponential single-fold Diophantine representation) they lead to extremely complex and computationally intensive formulas.

    \item In \cref{sectiontwoquadratic} we define the function $\chi(n)$ meaning the biggest natural number $s$ such that $s^2 \mid n$. We show that $\chi(n)$ equals the number of remainder classes $a \in \mathbb Z_n$ such that $a^2 \equiv 0 \mod n$. We also define the function $\omega(n)$ meaning the number of prime numbers dividing $n$. We recall that $2^{\omega(n) + 1}$ equals the number of remainder classes $a \in \mathbb Z_n$ such that $a^2 \equiv 1 \mod n$. This fact is also proved and used by the authors in Prunescu and Shunia \cite{prunescushunia2024primes}.
    \item In \cref{sectioncorrespondingfunctions} we construct Diophantine equations whose number of solutions in appropriated cubes equal $\chi(n)$, respectively $2^{\omega(n) + 1}$. Using the hypercube method, and in the second case the term for $\nu_2(n)$ given in \cref{sectionpreliminaries}, we construct arithmetic terms representing the functions $\chi(n)$ and $\omega(n)$. The term representing $\omega(n)$ has also been derived previously in \cite{prunescushunia2024primes}.
    \item A term $t(m,n)$ computing $\left \lfloor \sqrt[m]{n} \right \rfloor$ has been constructed in Prunescu and Sauras-Altuzarra \cite{prunescu2024numbertheoreticfunctions}. In \cref{sectionmthrootofn} we use an arithmetic term expressing the general exponentiation $x^y$ in order to construct a somewhat simpler term for $\left \lfloor \sqrt[m]{n} \right \rfloor$.
    \item In \cref{sectionfactoringterm} we combine the arithmetic terms $\gcd(m,n)$ and $n!$ given in \cref{sectionpreliminaries} with the terms $\chi(n)$ and $\omega(n)$ given in \cref{sectioncorrespondingfunctions} and with the term $\left \lfloor \sqrt[m]{n} \right \rfloor$ given in \cref{sectionmthrootofn} in order to construct an arithmetic term $T(n)$ with the following properties: (1) for all $n \in \mathbb N$, $T(n) \mid n$ and (2) if $n$ is composite, then $1 < T(n) < n$. This construction represents a more efficient alternative to the factorial-unwinding approaches of \cref{sectiondirectapplications}, avoiding repeated factorial constructions while maintaining the fixed-length, closed-form property. An alternative term $U(n)$ is also constructed using the same ideas and the explicit use of the truncated difference $x \dot{-} y$.
    \item \cref{sectionmoredetails} is an Appendix containing more details about the intricate arithmetic terms constructed in the sections \cref{sectioncorrespondingfunctions} and \cref{sectionmthrootofn}.
    \item \cref{sectionsourcecode} is an Appendix containing SageMath source code for verifying our ultimate arithmetic terms $T(n)$ and $U(n)$.
\end{enumerate}

\section{Preliminaries}\label{sectionpreliminaries}
We denote by $\mathbb{N}$ the set of natural numbers, including $0$.

{\bf Kalmar functions} are functions $f : \mathbb{N}^k \to \mathbb{N}$ that can be computed deterministically in iterated exponential time.

Mazzanti proved in \cite{mazzanti2002plainbases} that every Kalmar function can be expressed as an arithmetic term (See also Marchenkov \cite{marchenkov2007superposition} and Prunescu et al. \cite{prunescu2025minimalsubstitutionbasiskalmar}). Mazzanti's approach, known as the \textbf{hypercube method}, provides a complete characterization of Kalmar functions within the framework of arithmetic terms.

The function $\nu_2(n)$ returns the \textbf{$2$-adic order} of $n$, which is highest exponent of $2$ dividing $n$. The function $\HW(n)$ returns the \textbf{Hamming weight} of $n$, which is the number of $1$s in the binary representation of $n$.

We will make use of the following number theoretic arithmetic terms, which are proved in \cite{mazzanti2002plainbases, marchenkov2007superposition, prunescushunia2024gcd, prunescushunia2024primes}:
\begin{align*}
\binom{a}{b} &= \floor{\frac{(2^a+1)^a}{2^{ab}}} \bmod 2^a , \\
\binom{a}{b} &= \floor{\frac{2^{2(a+2)((a+1)^2+b+1)}}{2^{2(a+2)^2} - 2^{2(a+2)} - 1}} \bmod 2^{2(a+2)} ,\\
\gcd(a,b) &= \left ( \floor{\frac{2^{ ab(ab + a + b)}}{(2^{a^2 b} - 1)(2^{ab^2}-1)}} \bmod 2^{ab} \right ) - 1 \\
\nu_2(n) &= \floor{\frac{\gcd(n, 2^n)^{n+1} \bmod (2^{n+1}-1)^2}{2^{n+1}-1}} , \\
\HW(n) &= \nu_2\left(\binom{2n}{n}\right) .
\end{align*}

 Remark that Mazzanti gave a slightly more complicated $\gcd$-formula in \cite{mazzanti2002plainbases}.  The term $\HW(n)$ displayed above is a consequence of Kummer's Theorem, see Matiyasevich \citep[Appendix]{matiyasevich1993hilbert}.

Matiyasevich also shows that the generalized geometric series are represented by arithmetic terms:
\begin{align}
G_r (q, t) = \sum_{j=0}^{t-1} j^r q^j .
\end{align}
The corresponding term in the variables $q$ and $t$ will be called $G_r(q, t)$ as well.

Consider $a,b \in \mathbb{N} : 0 \leq a < 2^b$. We define the function
\begin{align}
\delta(a,b) := (2^b - 1)(2^b - a + 1) = 2^{2b} - 2^b a + a - 1.
\end{align}
It is proved in Mazzanti \cite{mazzanti2002plainbases} that:
\begin{align*}
\HW(\delta(a,b)) = \begin{cases}
2b, & \text{if } a = 0, \\
b, & \text{if } a \neq 0.
\end{cases}
\end{align*}
Consider integers $t, u > 0$ and a function $f : [0, t-1]^k \rightarrow \mathbb{N}$ that satisfies the condition:
$$\forall \,\vec a \in [0, t-1]^k\,\,\,\, 0 \leq f(\vec a) < 2^u.$$
We count the points with integer coordinates in the cube $[0, t-1]^k$ by the function:
\begin{align*}
\beta(\vec{a}) = a_1 + a_2 t(\vec{n}) + \cdots + a_k t(\vec{n})^{k-1}.
\end{align*}
We define the quantity:
\begin{align*}
M (f,t,u) = \sum_{\vec{a} \in \{0, \dots, t-1\}^k} 2^{2u \beta(\vec{a})} \delta(f(\vec{a}), u).
\end{align*}
If represented in binary form, the number $M$ consists of concatenated values $\delta(f(\vec a), u)$. In particular, the number of lattice points:
$$\{\vec a \in [0, t-1]^k : f(\vec a) = 0 \}$$
is exactly:
$$\frac{HW(M(f,t,u))}{u} - t^k.$$
Now, suppose that the function $f(\vec x)$ is a so called {\bf simple} exponential polynomial, i.e. a sum of monomials:
\begin{align*}
m(\vec x) = c v_1^{x_1} \cdots v_k^{x_k} x_1^{r_1} \cdots x_k^{r_k},
\end{align*}
where $r_1, \ldots, r_k \geq 0$, $v_1, \ldots, v_k \geq 1$ are integers, and $c \in \Z$. For summation we use the identity:
\begin{align*}
\sum_{\vec{a} \in \{0, \dots, t-1\}^k}
a_1^{r_1} v_1^{a_1}
\cdots a_k^{r_k} v_k^{a_k} = G_{r_1}(v_1, t)
\cdots G_{r_k}(v_k, t)
= \prod_{i=1}^k G_{r_i}(v_i, t) .
\end{align*}
We define the arithmetic term:
\begin{align} \label{TermA}
\mathcal{A}_k(m(\vec{x}), t, u)
&= -(2^{u} - 1) \cdot c
\cdot
\prod_{i=1}^k G_{r_i}(2^{2u t^{i-1}} v_i, t)
\end{align}
which depends only on $t$ and $u$, the argument $m(\vec x)$ being just a symbol, and not a value. Also, to the free term $c_0$ of the exponential polynomial, meaning the term $m(\vec x)$ corresponding to $v_1 = \cdots = v_k = 1$ and $r_1 = \cdots = r_k = 0$, it corresponds the term:
\begin{align} \label{TermC}
\CTerm_k(c_0, t, u)
= \frac{(2^{u} - c + 1)(2^{2u t^k} - 1)}{2^{u} + 1}.
\end{align}
It follows that in the case that $f(\vec x)$ is a simple exponential polynomial, the expression $M(f,t,u)$ introduced above
becomes an arithmetic term depending on $t$ and $u$, but also in all coefficients of $f(\vec x)$:
$$M(f,t,u) = \CTerm_k(c_0, t, u) + \sum \mathcal{A}_k(m(\vec{x}), t, u). $$
This fact will be applied in the following form. Suppose that one considers some arithmetic term $t(\vec n) > 0$ and some exponential Diophantine equation:
$$f(\vec n, \vec x) = 0$$
in parameters $\vec n$, such that:
\begin{enumerate}
    \item[(i)] All the coefficients of $f$ are given arithmetic terms in $\vec n$.
    \item[(ii)] Various basis of exponentiation $v_i$ are themselves
    arithmetic terms (or even other kind of functions) in $\vec n$.
    \item[(iii)] There is some arithmetic term $u(\vec n) > 0$ such that
    $$\forall \,\vec a \in [0, t(\vec n)-1]^k, \quad 0 \leq f(\vec n, \vec a) < 2^{u(\vec n)}.$$
\end{enumerate}

Altogether, we proved the following:
\begin{theorem}\label{proof:numberofsolutions}
For all $\vec n$,  the number of lattice points:
$$\{\vec a \in [0, t(\vec n)-1]^k: f(\vec n, \vec a) = 0 \}$$
is given by the following arithmetic term in $\vec n$:
$$\frac{ HW(M(f, \vec n, t(\vec n),u(\vec n))) }{u(\vec n)} - t(\vec n)^k.$$
\end{theorem}

 Arithmetic terms expressing the factorial function and the binomial coefficients go back to Julia Robinson \cite{robinson}. Such a term was constructed and used by Prunescu and Sauras-Altuzarra in \cite{prunescu2024factorial}. Here we may use the version from Prunescu and Shunia \cite{prunescushunia2024primes}:
\begin{align*}
n! &= \floor{ 8^{n^3} \Big/ \binom{8^{n^2}}{n} } .
\end{align*}
where we recall that $\binom{a}{b}$ is itself representable by an arithmetic term.

\section{Direct applications of the hypercube method}\label{sectiondirectapplications}

In this section we demonstrate how direct applications of the hypercube method presented in \cref{sectionpreliminaries} yield closed-forms $T_1(n)$, $T_2(n)$, $T_3(n)$, and $T_4(n)$ that compute various prime divisors and related quantities for composite integers $n$.

We start with a Lemma which was extensively proved by the authors in \cite{prunescushunia2024primes}. For the convenience of the reader, we will sketch the main ideas of the proof below.

\begin{lemma}\label{lemmafactorialsinglefold}
    There exists $k \in \mathbb N$ and a positive exponential polynomial $P(n, \vec x, a)$, simple in $\vec x \in \mathbb N^k$, such that for all $n, a \in \mathbb N$, one has:
    $$a = n! \longleftrightarrow \exists\,\vec x\,\,\,\, P(n, \vec x, a) =0. $$
    Moreover, this representation is single-fold, in the sense that if $a = n!$, then the tuple $\vec x \in \mathbb N^k$ that satisfies this exponential Diophantine equation is unique. Also, the exponential polynomial $P(n, \vec x, a)$ is a sum of squares.
\end{lemma}

\begin{proof}
    We start with the fact that
    $$n! = \floor{ 8^{n^3} \Big/ \binom{8^{n^2}}{n} } .$$
    We apply the definition for division with remainder and we find out that:
    $$a = n! \longleftrightarrow \exists\,x_1, x_2, x_3, x_4, x_5
    $$ $$ (x_1 - n^2)^2 + (x_2 - nx_1)^2 + \left(x_3 - \binom{2^{3x_1}}{n} \right)^2 +
    (2^{3x_2} - ax_3 -x_4)^2 + (x_3 - x_4 - x_5 - 1)^2 = 0.$$
    Indeed, $x_4 < x_3$ is the remainder of the division of $2^{3x_2}$ by the number $x_3$, and $a$ will be the quotient of this division with remainder. Now, by using the formula:
    $$\binom{a}{b} = \floor{\frac{2^{2(a+2)((a+1)^2+b+1)}}{2^{2(a+2)^2} - 2^{2(a+2)} - 1}} \bmod 2^{2(a+2)}, $$
    a similar positive single-fold exponential Diophantine representation is written for the relation:
    $$x_3 = \binom{2^{3x_1}}{n}.$$
    This representation is basically the conjunction of two divisions with remainder, one for the quotient and the other one for the remainder represented by the $\bmod$ operation. It will be also a sum of squares, and the variables can be chosen from $x_6$ upwards. If we replace the squared binomial defining $x_3$ with this sum of squares, we obtain the representation we were looking for.
\end{proof}

In the following lines, we show that there exists an arithmetic term $T_1(n)$ which computes the smallest prime divisor of $n$.

\begin{lemma}\label{lemmasmallestprimedivisor}
     If $n \in \mathbb N$ is a composite number, let $k$ be the number of tuples $(a, b, c, d, e, f) \in \mathbb N^6$ satisfying the following system of Diophantine equations:
    \begin{eqnarray*}
     (a+b+2) c &=& n \\
     d \cdot (a+b+1)! + 1 &=& en \\
     d + f &=& n
    \end{eqnarray*}
        Then $D = k+1$ is the smallest nontrivial divisor of $n$.

        If $n \in \mathbb N$ is a prime number, then the system has $n-1$ solutions.
\end{lemma}

\begin{proof}

 Consider a tuple $(a,b,c,d,e,f) \in \mathbb N^6$, and define $D:= a+b+2$. We will prove the following claims:

    {\it Claim 1: $D$ is the smallest non-trivial divisor of $n$. In particular, the value of $a+b$ is the same for all solutions of the system.}

    {\it Claim 2: $D = k+1$. }

    To prove {\it Claim 1}, we see that $D \geq 2$ and $D \mid n$, so $D \neq 1$ is indeed a divisor of $n$. By the second equation, the numbers $(D-1)!$ and $n$ are relatively prime, so there does not exist any smaller divisor different of $1$.

    To prove {\it Claim 2}, we observe that the values $D$, $c$, $d$, $e$, $f$ are uniquely determined by the composite number $n$. Indeed, the smallest nontrivial divisor $D$ is uniquely determined by $n$, and so is also $c := \floor{n/D}$. The number $d \leq n$ has the property that its class $d \bmod n$ multiplied with the class $(D-1)! \bmod n$ equals the class of $(n-1) \bmod n$. Such a unit $d \in \mathbb Z_n$ is uniquely determined by $n$. The numbers $e$ and $f$ are uniquely determined by $d$. It follows that, given $n$ composite, in the tuple $(a, b, c, d, e, f)$, only the natural numbers $a$ and $b$ can vary. The pair $(a,b)$ can take the following values:
    $$(D-2, 0); \,\,(D-3, 1); \,\, \dots; (0, D-2).$$
    As there are exactly $(D-2) + 1 = D-1$ such pairs, it follows that the number of tuples $(a, b, c, d, e, f) \in \mathbb N^6$ that satisfy the system, is $k = D-1$. Hence $D=k+1$.

    We also observe that if $n$ is a prime number, then $D = n$ is the only possible divisor. Then $(n-1)!+1$ is a multiple of $n$ by Wilson's Theorem, and the system is solvable. The number of solutions will be in this case $n-1$.
\end{proof}

\begin{theorem}\label{theoremsmallestprimedivisor}
    There is an arithmetic term $T_1$ such that for every natural number $n \geq 2$, the value $T_1(n)$ is equal with the smallest prime dividing $n$.
\end{theorem}

\begin{proof}
    As the smallest proper divisor of $n$ is always the smallest prime number dividing $n$, we have to encode only this information. The system of equations from Lemma \ref{lemmasmallestprimedivisor} is equivalent with the equation:
    $$((a+b+2) c - n)^2 + (d \cdot (a+b+1)! + 1 - en)^2 + (d + f - n)^2 = 0.$$
    We introduce a new variable $g$ for $(a+b+1)!$ and the corresponding monomial:
    $$((a+b+2) c - n)^2 + (g - (a+b+1)!)^2 + (dg  - en + 1)^2 + (d + f - n)^2 = 0.$$
    Using Lemma \ref{lemmafactorialsinglefold}, we know that $g = (a+b+1)!$ if and only if there exists a tuple $\vec x \in \mathbb N^k$ such that $P(a+b+1, \vec x, g) = 0 $. As this exponential polynomial is a sum of squares, we can substitute the squared binomial $(g - (a+b+1)!)^2$ with $P(a + b + 1, \vec x, g)$. So we obtain the exponential  polynomial:
    $$Q(n, a, b, c, d, e, f, g, \vec x) :=$$
    $$= ((a+b+2) c - n)^2 + P(a + b + 1, \vec x, g) + (dg  - en + 1)^2 + (d + f - n)^2.$$
    The exponential polynomial $Q$ is a positive function on $\mathbb N^{k+8}$ being a sum of squares. It is simple in $(a, b, c, d, e, f, g, \vec x)$. Also, for all composite numbers $n$ the number of $(k + 6)$-tuples which satisfy the equation $Q = 0$ is equal with $D-1$, where $D$ is the smallest nontrivial divisor of $n$; respectively with $n-1$ when $n$ is prime. By the hypercube method there is an arithmetic term $\theta(n)$ counting the solutions of the equation $Q=0$. We take $T_1(n):= \theta(n) + 1$.
\end{proof}

\begin{corollary}
    There is an arithmetic term $T_2$ such that for every composite number $n$, the value $T_2(n)$ is equal with the greatest nontrivial divisor of $n$, and for every prime number $p$, $T_2(p) = 1$.
\end{corollary}

\begin{proof}
    Let $t_1$ be the term constructed in Theorem \ref{theoremsmallestprimedivisor}. We define the arithmetic term $T_2$ by the formula:
    \begin{align*}
        T_2(n) = \floor{\frac{n}{T_1(n)}} .
    \end{align*}
\end{proof}

Further we show that there exists an arithmetic term $T_3(n)$ which, for every natural number $n$, computes the greatest prime divisor of $n$.

\begin{lemma}\label{lemmabiggestprimedivisor}
     If $n \in \mathbb N$ is a natural number, $n \geq 2$, let $k$ be the number of tuples $(a, b, c, d, e) \in \mathbb N^5$ satisfying the following system of Diophantine equations:
    \begin{eqnarray*}
    (a+b+2)c &=& n \\
    (a+b+1)! + 1 &=& d (a+b+2) \\
    ((a+b+2)!)^n &=& n e
    \end{eqnarray*}
        Then $D = k+1$ is the greatest prime divisor of $n$.

        If $n \in \mathbb N$ is a prime number, then the system has $n-1$ solutions.
\end{lemma}

\begin{proof}

 Consider a tuple $(a,b,c,d,e) \in \mathbb N^5$, and define $D:= a+b+2$. We will prove the following claims:

    {\it Claim 1: $D$ is the greatest prime divisor of $n$. In particular, the value of $a+b$ is the same for all solutions of the system.}

    {\it Claim 2: $D = k+1$. }

    To prove {\it Claim 1}, we observe that by the second equation, $(D-1)! + 1 = dD$, so according to Wilson's Theorem, $D$ is a prime number. From the first equation, $D \mid n$. From the last equation, all prime divisors of $n$ divide $D!$, so all these prime numbers are less or equal $D$. So $D$ is the greatest prime divisor of $n$.

    To prove {\it Claim 2}, we observe that the values $D$, $c$, $d$, $e$, $f$ are uniquely determined by the composite number $n$. Indeed, the greatest prime divisor $D$ is uniquely determined by $n$, and so is also $c := \floor{n/D}$. The number $d$ has the fixed value $\floor{[(D-1)! + 1] / D}$. Finally, $e$ fulfills $ (D!)^n = e n$. We observe that the number $e$ always exists if $D$ is the greatest prime divisor of $n$. Indeed, if $2 \leq p \leq D$ is any prime divisor of $n$, its exponent $\alpha$ in the prime number decomposition of $n$ satisfies $\alpha < n$. In the prime number decomposition of $(D!)^n$, the exponent of $p$ is $\beta \geq n > \alpha$. As this is true for every prime divisor of $n$, the number $e$ does exist.   It follows that, given $n \geq 2$ natural number, in the tuple $(a, b, c, d, e)$, only the natural numbers $a$ and $b$ can vary. The pair $(a,b)$ can take the following values:
    $$(D-2, 0); \,\,(D-3, 1); \,\, \dots; (0, D-2).$$
    As there are exactly $(D-2) + 1 = D-1$ such pairs, it follows that the number of tuples $(a, b, c, d, e, f) \in \mathbb N^6$ that satisfy the system, is $k = D-1$. Hence $D=k+1$.

    We also observe that if $n$ is a prime number, then $D = n$ is the only possible divisor. Then $(n-1)!+1$ is a multiple of $n$ by Wilson's Theorem, and the system is solvable. The number of solutions will be in this case $n-1$.
\end{proof}

\begin{theorem}\label{theorembiggestprimedivisor}
        There is an arithmetic term $T_3$ such that for every natural number $n \geq 2$, the value $T_3(n)$ is equal with the greatest prime dividing $n$.
\end{theorem}

\begin{proof}
    The proof is similar with the proof of Theorem \ref{theoremsmallestprimedivisor} and consists in writing down a positive exponential Diophantine polynomial which contains $n$ as a parameter and is simple in its variables, such that for every value of $n \geq 2$, this exponential polynomial has exactly $D - 1$ many solutions, where $D$ is the greatest prime dividing $n$.

    As we don't want a double occurrence of the factorial operation, we introduce new  unknowns $f$, $g$, $h$ and we write the system of equations from Lemma \ref{lemmabiggestprimedivisor} in the form:
    $$ ((a+b+2)c-n)^2 + (f - (a+b+1)!)^2 + (f+1-d(a+b+2))^2 + (g - f(a+b+2))^2 + (h-g^n)^2 + (h - ne)^2 = 0. $$
    Of course, in every solution $f = (a+b+1)!$,
    $g = (a+b+2)!$ and $h = ((a+b+2)!)^n$. Like in the proof of Theorem \ref{theoremsmallestprimedivisor}, we replace the squared binomial $(f-(a+b+1)!)^2$ with the sum of squares $P(a+b+1, \vec x, f)$. We still have to get rid of non-simple squared binomial
    $(h-g^n)^2$. In the section \ref{sectionmthrootofn} we will prove that:
    $$ g^n = 2^{3n^2g} \bmod (2^{3ng} - g) $$
    for all $n, g \geq 1$. Based on this fact, in the same section we will find a simple in $\vec y$ exponential polynomial sum of squares
    $ E(a, b, \vec y, c) $ such that:
    $$c = a^b \longleftrightarrow \exists \, \vec y \,\,\,\, E(a, b, \vec y, c) = 0.$$
    Moreover, this representation is single-fold. The sum of squares $E(g, n, \vec y, h)$ replaces the squared binomial $(h - g^n)^2$. We obtain the exponential polynomial:
    $$R(n, a, b, c, d, e, f, g, h, \vec x, \vec y):=$$
    $$= ((a+b+2)c-n)^2 + P(a+b+1, \vec x, f) + (f+1-d(a+b+2))^2 + (g - f(a+b+2))^2 + E(g, n, \vec y, h) + (h - ne)^2.$$
    The number of solutions $k$ of the equation $R=0$ are computed by a term $\rho(n)$ obtained by the hypercube method. We take $T_3(n) := \rho(n) + 1$.
\end{proof}

Recall that the OEIS sequence A007917 lists the largest prime less than or equal to $n$, see \cite{oeisA007917}.

\begin{corollary}
    There is an arithmetic term $T_4(n)$ which for all $n \geq 2$ outputs the greatest prime less than or equal to $n$.
\end{corollary}

\begin{proof}
    The arithmetic term $T_4(n) := T_3( n! )$ fulfills that condition. As the binomial coefficients are easier to compute, we may also consider the term:
    $$T_4(n) := T_3 \left ( \binom{n}{\floor{n/2}}\right ).$$
\end{proof}

While the arithmetic terms $T_1(n)$, $T_2(n)$, $T_3(n)$, and $T_4(n)$ constructed in this section demonstrate that integer divisors exhibit fixed-length elementary closed-forms, the factorial-unwinding approach leads to enormous computational complexity. Each of these terms requires unwinding the factorial function multiple times through complex exponential Diophantine representations, resulting in formulas with numerous variables and exponentially large intermediate values.

Recognizing these limitations, we develop in the remainder of the paper an alternative arithmetic term $T(n)$ based on fundamentally different ideas. Rather than unwinding factorials, the construction in \cref{sectionfactoringterm} combines arithmetic terms for the quadratic residue-based invariants $\chi(n)$ (the largest square divisor) and $\omega(n)$ (the number of distinct prime divisors), developed in \cref{sectiontwoquadratic}--\cref{sectioncorrespondingfunctions}, together with a term for the $m$-th root function constructed in \cref{sectionmthrootofn}. While $T(n)$ remains computationally infeasible for large $n$ in absolute terms, it avoids the repeated factorial constructions that make the terms of this section comparatively more complex.

\section{Two quadratic equations modulo n}\label{sectiontwoquadratic}

In this section, $\mathbb Z_n = \mathbb Z / n \mathbb Z$ means the ring consisting of the remainder classes modulo $n \geq 2$. By the Chinese Remainder Theorem, if the prime-number decomposition of $n$ is $n = p_1^{k_1}\dots p_t^{k_t}$, then the following rings are isomorphic:
\begin{align*}
\Z/n\Z \cong \Z/p_1^{k_1}\Z \times \cdots \times \Z/p_t^{k_t }\Z.
\end{align*}
We consider the function $\chi : \mathbb N \setminus\{0\} \rightarrow \mathbb N$ defined as:
$$\chi(n) = \max \{ s \in \mathbb N : s^2 | n \}.$$

\begin{theorem}\label{modeqchi}
    For all $n \geq 1$,
    $$\chi(n) = |\{ a \in \mathbb Z_n : a^2 = 0 \textup{  in } \mathbb Z_n \} |$$
\end{theorem}

\begin{proof}
Consider $n = p^{2k}$ where $p$ is a prime and $k \geq 1$.
Let $x \in \{0, 1, \dots, p^{2k} - 1\}$ be the representative of a solution of the equation $x^2 = 0$ in $\mathbb Z_{p^{2k}}$.
Then
\begin{align*}
   p^{2k} \mid x^2 \iff p^k \mid x .
\end{align*}
A multiple of $p^k$ which is smaller than $p^{2k}$ is of the form
\begin{align*}
s \cdot p^k < p^{2k} \iff s < p^k.
\end{align*}
So there are exactly $p^k$ many solutions, corresponding to all values $s \in \{0, \dots, p^k-1\}$.

Next, consider $n = p^{2k + 1}$, with $p$ prime and $k \geq 1$. Let $x \in \{0, 1, \dots, p^{2k + 1} - 1\}$ be the representative of a solution of the equation $x^2 = 0$ in $\mathbb Z_{p^{2k + 1}}$.
Then
\begin{align*}
    p^{2k+1} \mid x^2 \iff p^{k + 1} \mid x .
\end{align*}
A multiple of $p^{k + 1}$ which is smaller than $p^{2k + 1}$ is of the form
\begin{align*}
s \cdot p^{k + 1} < p^{2k+1} \iff s < p^k .
\end{align*}
So there are again exactly $p^k$ many solutions.

If $n = p$ is a prime, then the condition $a^2 \equiv 0 \pmod{p}$ implies $p \mid a$. Since $a$ must be in $\{0, 1, \dots, p - 1\}$, the only value that satisfies this condition is $a = 0$. Hence, there is exactly one solution.

Finally, consider the general case where:
\begin{align*}
n = p_1^{2k_1 + b_1} \cdots p_t^{2k_t + b_t} \quad \text{ with } b_i \in \{0, 1\}.
\end{align*}
By the Chinese Remainder Theorem, there exists an isomorphism of rings:
\begin{align*}
\Z/n\Z \cong \Z/(p_1^{2k_1 + b_1})\Z \times \cdots \times \Z/(p_t^{2k_t + b_t})\Z.
\end{align*}
The total number of solutions for $x^2 \equiv 0 \pmod{n}$ is the product of the numbers of solutions for each prime power component:
\begin{align*}
\prod_{i=1}^t p_i^{k_i}.
\end{align*}
This product is precisely the square root of the largest square dividing $n$:
\begin{align*}
\sqrt{p_1^{2k_1} \cdots p_t^{2k_t}} = \chi(n).
\end{align*}
\end{proof}

We also consider the function $\omega : \mathbb N \setminus \{0\} \rightarrow \mathbb N$ defined as:
$$\omega(n) = | \{ p \textup{ prime }:\,\, p \mid n \}|.$$

The following Theorem has been also proved and used by the authors in \cite{prunescushunia2024primes}. We sketch it here only to be compared with the proof above.

\begin{theorem}\label{eqmodomega}
    For all $n \geq 1$,
    $$2^{\omega(n) + 1} = |\{ a \in \mathbb Z_{4n} : a^2 = 1 \textup{  in } \mathbb Z_{4n} \} |.$$
\end{theorem}

\begin{proof} If $p$ is an odd prime and $k \geq 1$, then the equation $x^2 = 1$ has in $\mathbb Z_{p^k}$ only the solutions $\{1, -1\}$. The same is true for the ring $\mathbb Z_4$. In $\mathbb Z_{2^k}$ with $k \geq 3$, there are four solutions: $\{1, 2^{k-1} - 1, 2^{k-1} + 1, 2^k -1\}$. Now suppose that $n$ is odd and has $k$ prime divisors. The number of solutions of the equation $x^2 = 1$ in $\mathbb Z_{4n}$ will be then $2 \cdot 2^k = 2^{k+1}$. If $n$ is even and has $k$ prime divisors, the number of solutions of $x^2 = 1$ in $\mathbb Z_{4k}$ is $2^2 \cdot 2^{k-1} = 2^{k+1}$.
\end{proof}

\section{The corresponding functions}\label{sectioncorrespondingfunctions}

First we show:

\begin{lemma}
$$\chi(n) = |\{ a \in \mathbb Z_n : a^2 = 0 \textup{  in } \mathbb Z_n \} | = |\{ (x,y) \in [0, n-1]^2 : x^2 = ny\} | $$
\end{lemma}

\begin{proof}
    Indeed, if $x^2 = 0$ in $\mathbb Z_n$, then one can choose a representative $0 \leq x < n$ for this solution. As $x^2 = ny$ for some $y \in \mathbb Z$, we have that
    $$0 \leq y < \frac{n^2}{n} = n.$$
    Also, if the pair $(x,y) \in [0, n-1]^2$ satisfies $x^2 = ny$ then for the corresponding class modulo $n$ one has $x^2 = 0$.
\end{proof}

Working with the conventions introduced in Section \ref{sectionpreliminaries}, we count how many zeros $(x,y)$ has the function:
$$f(n, x, y) = (x^2 - ny)^2$$
in the square $[0, t(n)-1]^2$ with $t(n) = n$.

\begin{lemma}
    For all $n \geq 1$ and for all $(x,y) \in [0, n-1]^2$
    one has:
    $$0 \leq (x^2 - ny)^2 < 2^{n+4}.$$
\end{lemma}

\begin{proof}
    We see that the expression $(x^2 - ny)^2$ takes the maximal value $n^2(n-1)^2$ for the pair $(0,n-1)$. By inspecting values, we observe that for all $n \geq 0$:
    $$n^2(n-1)^2 < 2^{n+4}$$
\end{proof}
According to this lemma, we can take $u(n) = n+4$.

\begin{theorem} \label{proof:chi}
\begin{align*}
\forall n \in \Z^+, \quad \chi(n) = \frac{\HW(M(f, n,  n, n+4))}{n+4} - n^2 .
\end{align*}
\end{theorem}

An arithmetic term for $\omega(n)$  was computed by Prunescu and Shunia in \cite{prunescushunia2024primes}. Its construction is analogous with the term for $\chi(n)$ as the following steps show. As the Lemmas have similar proofs with the corresponding Lemmas above, we do not display these proofs.

\begin{lemma}
    $$2^{\omega(n) + 1} = |\{ a \in \mathbb Z_{4n} : a^2 = 1 \textup{  in } \mathbb Z_{4n} \} |  = |\{ (x,y) \in [0, 4n-1]^2 : x^2 = 4ny + 1\} |.$$
\end{lemma}

So we have to count how many zeros $(x,y)$ has the function:
$$g(n, x, y) = (x^2 - ny - 1)^2$$
in the square $[0, t(n)-1]^2$ with $t(n) = n$. After constructing a term that expresses this number of solutions, we will replace the argument $n$ with $4n$.

\begin{lemma}
    For all $n \geq 1$ and for all $(x,y) \in [0, n-1]^2$
    one has:
    $$0 \leq (x^2 - ny - 1)^2 < 2^{n+4}.$$
\end{lemma}
In conclusion, one can take $u(n) = n+4$.

\begin{theorem} (Prunescu and Shunia, \citep[Theorem 5.1]{prunescushunia2024primes}) \label{proof:omegasolutions}
\begin{align*}
\forall n \in \Z^+, \quad
\omega(n) = \nu_2\left( \frac{\HW(M(g, 4n,  4 n, 4 n + 4))}{4 n + 4} - 16 n^2 \right) - 1 .
\end{align*}
\end{theorem}
Here we recall that $\nu_2(c)$ is expressed by a term in $c$.

\section{The m-th root of n}\label{sectionmthrootofn}

Another function needed to construct the factoring formula is $ \left\lfloor \sqrt[m]{n} \right\rfloor $. This function is a weak point in our construction because the corresponding term is extremely complex.

{\bf Alternatively}, we might consider $ \left\lfloor \sqrt[m]{n} \right\rfloor $ as a {\bf primitive} function, and to present the factoring term as a composition and superposition of the operations of addition, modified subtraction, multiplication, division with remainder, arbitrary root extraction and exponentiation.

The term we present here follows the construction in Prunescu and Sauras-Altuzarra \citep[Theorem 8.3]{prunescu2024numbertheoreticfunctions}. We sketch here the construction of a slightly shorter term, with products of only $5$ G-functions and fewer variables.

We observe that:
$$\left\lfloor \sqrt[m]{n} \right\rfloor + 1 = |\{x \in [0,n] : x^m \leq n\}|.$$
The condition $x^m \leq n$ is equivalent with $n = x^m + y$ for some  integer $y \geq 0$.
This definition leads to a family of arithmetic terms, another one for every $m$. But we need just one term in variables $x$ and $m$. Therefore, we need a two-variable term expressing $x^m$. The classical arithmetic term representation for this function has been given in \cite{marchenkov2007superposition} by Marchenkov:
$$x^m = 2^{(xm+x+1)m} \bmod \left ( 2^{xm+x+1} - x \right ) $$
As this operation has to be unwinded in a logic formula, we prefer a simpler expression, that contains less additions.
\begin{lemma} For all $m > 0$ and $x \in \mathbb N$ one has
$$x^m = 2^{3m^2 x} \bmod ( 2^{3mx} \dot{-} x ).$$
\end{lemma}
\begin{proof}
    For $m, x > 0$ we have that $2mx > 1$, so $3mx-1 > mx$, hence:
    $$2^{3mx-1} > 2^{mx} > x^m.$$
    So it follows that:
    $$2^{3mx}> 2x^m \geq x^m + x,$$
    $$2^{3mx} - x > x^m.$$
    So,
    $$2^{3m^2x} = ( x + 2^{3mx} - x)^m,$$
    $$2^{3m^2 x} \bmod ( 2^{3mx} - x ) = x^m.$$
    We observe that for $x = 0$ one gets $1 \bmod 1 = 0 = 0^m$. The formula works also for $x > 0$ and $m = 0$ because then $1 \dot{-} x = 0$ and by convention $1 \bmod 0 = 1$. But we will not use this situation, because $\sqrt[0]{n}$ does not make any sense.
\end{proof}
The exponentiation formula can be written as:
\[x^m = r \iff \exists \, (q, d) \in \mathbb{N}^2 :
    (2^{3m^2 x} = 2^{3mx} q - xq + r) \wedge (r + d + 1 = 2^{3mx} ).
\]
So we have to count the solutions $(x,y,r,q,d) \in \mathbb N^5$ for the exponential Diophantine equation:
$$(n-r-y)^2 + (2^{3m^2 x} - 2^{3mx} q + xq - r)^2 +(2^{3mx} - x - r - d - 1)^2 = 0$$
as a function of the parameters $m, n$.

As this point, we can transform this equation such that we diminish the number of monomials of the developed form.
The equation
$$(n-r-y)^2 = 0$$
has the same solutions $(r,y)$ like the equation
$$(2^n - 2^{r+y})^2 = 0$$
and can be substituted by the later. As well, if we introduce a new variable $z$ defined by the equation
$$(2^{3mx}  - z)^2 = 0$$
then, instead of
$$(z - x - r - d - 1)^2 = 0,$$
we can write only:
$$( 2^z - 2^{x + r + d + 1} )^2 = 0.$$
Finally, we introduce still other three variables $w$, $g$ and $s$ defined by the equations:
$$(w - xq)^2 = 0,$$
$$(2^{3mx} q - g )^2 = 0,$$
$$( 2^{3m^2 x} - s)^2 = 0.$$
Now the middle equation
$$(2^{3m^2 x} - 2^{3mx} q + xq - r)^2 =0$$
becomes
$$(s - g + w - r)^2 = 0$$
which is equivalent with:
$$( 2^{s + w} - 2^{q+r} )^2 = 0.$$
To sum up, we count the number of solutions $(x,y,z,w,g,s,r,q,d) \in \mathbb N^9$ satisfying
$$ (2^{3mx} - z)^2 + (2^{3mx} q - g )^2 + (w - xq)^2 + (2^{3m^2x} - s)^2 + $$$$ + (2^n - 2^{r+y})^2 + ( 2^{s + w} - 2^{z+r} )^2  + ( 2^z - 2^{x + r + d + 1} )^2 = 0.$$

The monomial expansion for this sum of squares is:
\begin{align*}
0 =
& 2^{2 n}
- 2 \cdot 2^{r} 2^{s} 2^{w} 2^{z}
- 4 \cdot 2^{d} 2^{r} 2^{x} 2^{z}
+ q^{2} x^{2}
+ 4 \cdot 2^{2 d} 2^{2 r} 2^{2 x}
- 2 \cdot 2^{n} 2^{r} 2^{y}
- 2 \cdot 2^{3 m x} g q
+ 2^{6 m x} q^{2} \\
& - 2 q w x
+ 2^{2 s} 2^{2 w}
+ 2^{2 r} 2^{2 y}
+ 2^{2 r} 2^{2 z}
+ g^{2}
- 2 \cdot 2^{3 m^{2} x} s
+ s^{2}
+ w^{2}
- 2 \cdot 2^{3 m x} z
+ z^{2}
+ 2^{6 m^{2} x} \\
& + 2^{6 m x}
+ 2^{2 z} .
\end{align*}

\begin{lemma} \label{lemma:solutions}
    For all $m,n \geq 1$, the solutions $(x,y,z,w, g, s,r,q,d)$ of the equation $$h(m,n,x,y,z,w,g,s,r,q,d) = 0$$ belong to the cube:
    $$[0, 2^{6n^3m^2}]^9.$$
\end{lemma}

\begin{proof}
    This bound is based on the fact that for a division with remainder, both the quotient and the remainder are smaller than the number to be divided. We see that:
    $$ q < 2^{3m^2n},$$
    $$r \leq n,$$
    $$x \leq n,$$
    $$ y \leq n, $$
    $$ d < 2^{3mn},$$
    $$ z = 2^{3mx} \leq 2^{3mn},$$
    $$ w = xq \leq n \cdot 2^{3mn}.$$
    $$ g =  2^{3mx} q \leq n \cdot 2^{3mn} \cdot 2^{3mn} = n \cdot 2^{6mn},$$
    $$ s = 2^{3m^2x} \leq 2^{3m^2n}.$$

    So all unknowns are smaller or equal than
    $$n \cdot 2^{6n^2m^2} < 2^{6n^3m^2}.$$
    \end{proof}

Accordingly, we take $t(m,n) = 2^{6n^3m^2}. $

In order to generate less monomials, we replace $( 2^s - 2^{z+w+r} )^2$ with $( 1 - 2^{z+w+r-s} )^2$ and $( 2^z - 2^{x + r + d + 1} )^2$ with $( 1 - 2 \cdot 2^{x + r + d - z} )^2$. But these expressions take some values which are not integers if $z, w, r, s \in [0, t(m,n)]$, respectively $x, r, d, z \in [0, t(m,n)]$. In order to prevent this,
we will consider the squared binomials: $( 2^{2t(m,n)} - 2^{2t(m,n)} \cdot 2^{z+r-s-w} )^2$, respectively with $ ( 2^{t(m,n)} - 2 \cdot 2^{t(m,n)} \cdot 2^{x + r + d - z} )^2$. It is clear that these squared binomials take always positive integers as values, and that they have the same zeros as the initial squared binomials. From every binomial, after squaring, the monomial depending on $m, n$ only will contribute to the free term, so instead of $21$ monomials, we get only $19$. We denote with $h'(m,n,t, x,y,z,w,g, s,r,q,d)$ the function:
$$h'(m,n,t, x,y,z,w,g, s,r,q,d) =$$
$$ (2^{3mx} - z)^2 + (2^{3mx} q - g )^2 + (w - xq)^2 + (2^{3m^2x} - s)^2 + $$$$ + (2^n - 2^{r+y})^2 + ( 2^{2t} - 2^{2t} \cdot 2^{z+r-s-w} )^2  + ( 2^{t} - 2 \cdot 2^{t} \cdot 2^{x + r + d - z} )^2.$$

\begin{lemma}
    For all $t > m,n \geq 1$, and for all $(x,y,z,w,s,r,q,d) \in [0, t]^9 $, one has that:
    $$0 \leq h'(m,n,t,x,y,z,w,g, s,r,q,d) < 2^{6t^3 + 5},$$
    and all values of $h$ are integers.
\end{lemma}

\begin{proof}
    In the expression $h(m,n,t,x,y,z,w,g, s,r,q,d)$ we replace every minus sign with a plus sign. This operation majorizes the expression. Now we replace all letters, unknowns and parameters, with $t$. It is clear that the fourth squared binomial contains the biggest quantity, which is
    $$2^{3t^3}.$$
    So the content of any of the squared terms is smaller than
    $$2 \cdot 2^{3t^3}, $$
    and any of the seven squared terms is smaller than
    $$4 \cdot 2^{6t^3}.$$
    If we replace all seven terms with the biggest term, we find that their sum is smaller than:
    $$7 \cdot 4 \cdot 2^{6t^3} = 28 \cdot 2^{6t3} \leq 2^{6t^3 + 5}.$$
\end{proof}

So we can take $u(t) = 6t^3 + 5$. Also, for constructing the arithmetic term, we have to replace the parameter $t$ occurring inside the expression $h$ with $t(m,n)$. The expression used will be the following:
$$h''(m,n,x,y,z,w,g,s,r,q,d) =$$
$$=(2^{3mx} - z)^2 + (2^{3mx} q - g )^2 + (w - xq)^2 + ( 2^{3m^2x} - s)^2 + $$$$ + (2^n - 2^{r+y})^2 + ( 2^{2t(m,n)} - 2^{2t(m,n)} \cdot 2^{z+r-s-w} )^2  + ( 2^{t(m,n)} - 2 \cdot 2^{t(m,n)} \cdot 2^{x + r + d - z} )^2.$$

The monomial expansion is:
\begin{align*}
& h''(m,n,x,y,z,w,g,s,r,q,d) =
2^{2  n}
+ q^{2} x^{2}
- 2 \cdot 2^{n} 2^{r} 2^{y}
- \frac{4 \cdot 2^{2 \cdot 2^{6  m^{2} n^{3}}} 2^{d} 2^{r} 2^{x}}{2^{z}}
- 2 \cdot 2^{3  m x} g q
+ 2^{6  m x} q^{2} \\
& - 2  q w x + 2^{2  r} 2^{2  y}
+ \frac{4 \cdot 2^{2 \cdot 2^{6  m^{2} n^{3}}} 2^{2  d} 2^{2  r} 2^{2  x}}{2^{2  z}}
+ g^{2}
+ s^{2}
+ w^{2}
+ z^{2}
- 2 \cdot 2^{3  m^{2} x} s
- 2 \cdot 2^{3  m x} z
+ 2^{6  m^{2} x} + 2^{6  m x} \\
& + 2^{4 \cdot 2^{6  m^{2} n^{3}}}
+ 2^{2 \cdot 2^{6  m^{2} n^{3}}}
- \frac{2 \cdot 2^{4 \cdot 2^{6  m^{2} n^{3}}} 2^{r} 2^{z}}{2^{s} 2^{w}}
+ \frac{2^{4 \cdot 2^{6  m^{2} n^{3}}} 2^{2  r} 2^{2  z}}{2^{2  s} 2^{2  w}}
\end{align*}

\begin{theorem}\label{newmrootofn}
\begin{align*}
\forall n \in \Z^+, \quad
 \left\lfloor \sqrt[m]{n} \right\rfloor = \left ( \frac{\HW(M(h'', m,  n,  t(m, n), u(t(m, n))))}{u(t(m,n))} - t(m,n)^9 \right) - 1 ,
\end{align*}
\end{theorem}

Only one remark about how to apply the construction method in this case. Here, not only some coefficients of the exponential Diophantine equation are arithmetic terms depending on the parameters, but also some bases of exponentiation. Such a monomial is:
$$ 2^{6m^2x} =  {\left( 2^{6m^2} \right ) }^x.  $$
The nontrivial contributions of such monomials go in  functions of type $G_0$.

In \cite{shunia2024polynomial}, Shunia conjectured the formula:

\begin{conjecture} (Shunia, \cite{shunia2024polynomial}) \label{proof:roots}
Let $n,m \in \mathbb{N}$ such that $n > 2$, $\floor{\log_2(n)} + 1 \geq m > 1$, and $\not\exists k \in \mathbb{N} \, | \, k^m = n$. Then
\begin{align*}
\floor{\sqrt[m]{n}}
&= \floor{\frac{(n^{2nm} + 1)^{2nm+1} \bmod{(n^{2nm^2}-n)}}{(n^{2nm} + 1)^{2nm} \bmod{ (n^{2nm^2}-n)}} - 1} .
\end{align*}
\end{conjecture}
If proven, this term could significantly simplify the symbolic representation of our main result.

\section{A formula for factoring n}\label{sectionfactoringterm}

In this section, we synthesize the arithmetic terms developed in \cref{sectiontwoquadratic}--\cref{sectionmthrootofn} to construct our main result: An arithmetic term $T(n)$ that yields a non-trivial divisor of any composite integer $n > 1$. Unlike the factorial-unwinding approaches of \cref{sectiondirectapplications}, this construction combines the quadratic residue-based invariants $\chi(n)$ (the largest square divisor) and $\omega(n)$ (the number of distinct prime divisors) with root extraction and greatest common divisor operations. By avoiding repeated factorial constructions, $T(n)$ achieves a comparatively more efficient structure, though it remains computationally infeasible for large $n$ in absolute terms. We also present an alternative formulation $U(n)$ using the truncated difference operation.

\begin{lemma} \label{proof:divisorrootbounds}
Let $n = p_1 p_2 \cdots p_k$ be a composite, squarefree integer with $k = \omega(n)$ distinct prime factors, where $p_1 < p_2 < \dots < p_k$. Then $p_1 \leq \floor{\sqrt[k]{n}} < p_k$.
\end{lemma}

\begin{theorem} \label{proof:sffactor}
Let $n \in \mathbb{N}$ such that $n$ is composite and squarefree. Then a non-trivial divisor of $n$ is given by:
\begin{align*}
q = \gcd(n, \floor{\sqrt[\omega(n)]{n}}!) .
\end{align*}
\end{theorem}
\begin{proof}
Consider $n = p_1 p_2 \cdots p_k$ which is squarefree and the product of $k$ distinct primes. Now, let
\begin{align*}
r = \floor{\sqrt[\omega(n)]{n}} .
\end{align*}

By \cref{proof:divisorrootbounds}, we know that
\begin{align*}
   \min(p_1,\ldots,p_k) \leq r < \max(p_1,\ldots,p_k) .
\end{align*}
It follows immediately that
\begin{align*}
q = \gcd(n, r!) = \gcd(n, \floor{\sqrt[\omega(n)]{n}}!)
\end{align*}
is a non-trivial divisor of $n$.
\end{proof}

\begin{theorem} \label{proof:divisor1}
Let $n \in \mathbb{N}$ such that $n>1$ is composite. Then a non-trivial divisor of $n$ is given by:
\begin{align*}
q = \gcd\left( \frac{n}{\chi(n)}, \floor{\sqrt[\omega(n)]{n}}! \right) :=T(n) .
\end{align*}
\end{theorem}

\begin{proof}
    If $\chi(n) = 1$ then $n$ is square-free, the expression gets the form
    $$\gcd\left( n, \floor{\sqrt[\omega(n)]{n}}! \right)  $$
    and we apply  \cref{proof:sffactor}.  If $\chi(n) \neq 1$ then $n/\chi(n) < n$, so $q < n$. Also $n/\chi(n)$ is a divisor of $n$ by \cref{proof:chi}. We must show that $q \neq 1$. Let $p_1$ be the smallest prime that divides $n$. The exponent of $p_1$ in the prime-number decomposition of $n$ is $2k+b$ with $b \in \{0,1\}$ such that if $b = 0$ then $k \geq 1$.
    The exponent of $p_1$ in $\chi(n)$ is $k$. The exponent of $p_1$ in $n/\chi(n)$ is $k+b \geq 1$, so
    $$p_1 \mid \frac{n}{\chi(n)}.$$
    Now we look at the number $\floor{\sqrt[\omega(n)]{n}}$. Denote $\omega(n)$ with $\omega$. Then:
    $$n = p_1^{\alpha_1}\dots p_\omega^{\alpha_\omega} $$
    where all $\alpha_i \geq 1$. It follows:
    $$n \geq p_1^\omega,$$
    $$\floor{\sqrt[\omega(n)]{n}} \geq p_1,$$
    $$p_1 \mid \floor{\sqrt[\omega(n)]{n}} !$$
    So $p_1 \mid q$, which implies that $q \neq 1$.
\end{proof}

\begin{remark} \label{proof:divisor} \rm
Let $n \in \mathbb{N}$ such that $n>1$ is composite. Then a non-trivial divisor of $n$ is also given by:
\begin{align*}
q = (2 \dot{-} \chi(n)) \cdot \gcd(n, \floor{\sqrt[\omega(n)]{n}}!) + (1 \dot{-} (2 \dot{-} \chi(n))) \cdot \chi(n) :=U(n).
\end{align*}
\end{remark}
\begin{proof}
This remark is another way to aggregate the results given in \cref{proof:sffactor} and \cref{proof:chi}. Following Marchenkov \cite{marchenkov2007superposition}, we use a somewhat artificial feature of the modified difference:
$$a \dot{-} b = 0 \iff a \leq b.$$
So $2 \dot{-} \chi(n) = 1$ if and only if $\chi(n) = 1$ if and only if $n$ is square-free; otherwise $2 \dot{-} \chi(n) = 0$. Also, $1 \dot{-} (2 \dot{-} \chi(n)) = 1$ if and only if $\chi(n) > 1$ if and only if $n$ is divisible by nontrivial squares, otherwise $1 \dot{-} (2 \dot{-} \chi(n)) = 0$.
\end{proof}

As the functions $\gcd(a, b)$, $\chi(n)$, $\omega(n)$, $\floor{\sqrt[m]{n}}$ and $k!$ are all represented by arithmetic terms, our both formulas for a divisor of $n$ are arithmetic terms.

To compute some examples,
$$T(10) = \gcd\left( \frac{10}{\chi(10)}, \floor{\sqrt[2]{10}}! \right)
= \gcd\left( \frac{10}{1}, \floor{\sqrt[2]{10}}! \right) = \gcd(10, 6) = 2,$$
$$T(20) = \gcd\left( \frac{20}{\chi(20)}, \floor{\sqrt[2]{20}}! \right)
= \gcd\left( \frac{20}{2}, \floor{\sqrt[2]{20}}! \right) = \gcd(10, 24) = 2,$$
$$T(50) = \gcd\left( \frac{50}{\chi(50)}, \floor{\sqrt[2]{50}}! \right)
= \gcd\left( \frac{50}{5}, \floor{\sqrt[2]{50}}! \right) = \gcd(10, 5040) = 10.$$

\section{Appendix: More details about the formulas}\label{sectionmoredetails}

For the function $\chi(n)$, one has:
\begin{align*}
\forall n \in \Z^+, \quad \chi(n) = \frac{\HW(M(f, n, t, u))}{u} - t^2 ,
\end{align*}
where  $t = n$, $u = n+4$, and
\begin{align*}
M(f, n, t, u) = \mathcal{C}_2(0, t, u)
+ \mathcal{A}_2(x_1^4, t, u)
+ \mathcal{A}_2(-2n x_1^2 x_2, t, u)
+ \mathcal{A}_2(n^2 x_2^2, t, u) .
\end{align*}
Indeed,
$$(x^2 - ny)^2 = x^4 - 2n x^2y + n^2y^2.$$
For example,
$$ \mathcal{A}_2(x_1^4, t, u) = - (2^u-1)G_4(2^{2u}, t) G_0(2^{2ut},t) = - (2^{n+4} -1)
G_4(2^{2(n+4)}, n) G_0(2^{2n(n+4)}, n).$$

For the function $\omega(n)$, one has:

\begin{align*}
\forall n \in \Z^+, \quad
\omega(n) = \nu_2\left( \frac{\HW(M(g, 4n, t, u))}{u} - t^2 \right) - 1 .
\end{align*}
where $t=4n$, $u = 4n+4$, and
\begin{align} \label{TermM}
\begin{array}{llll}
M(g, n, t, u)
&= \CTerm_2(1, t, u)
&+ \mathcal{A}_2(x_1^4, t, u)
&+ \mathcal{A}_2(-2 x_1^2, t, u) \\
&+ \mathcal{A}_2(-2n x_1^2 x_2, t, u)
&+ \mathcal{A}_2(n^2 x_2^2, t, u)
&+ \mathcal{A}_2(2n x_2, t, u) .
\end{array}
\end{align}
Indeed,
$$(x^2 - ny - 1)^2 = n^4 + n^2y^2 + 1 -2nx^2y +2ny -2x^2.$$
For example,
$$\mathcal{A}_2(-2n x_1^2 x_2, t, u) = 2n (2^u -1)
G_2(2^{2u}, t) G_1(2^{2ut},t) = 2n (2^{4n+4} - 1) G_2(2^{8(n+1)}, 4n) G_1(2^{16n(n+1)},4n).$$

\begin{lemma}
\label{proof:rootsprunescusauras}
The function $ \left\lfloor \sqrt[m]{n} \right\rfloor $ (for integer arguments $ n \geq 1 $ and $ m \geq 2 $) can be represented by the arithmetic term $$ \HW ( M ( m , n ) ) / ( 2^{nm^2 + nm + 2} + 2 (nm^2 + nm) + 9 ) - 2^{7 (nm^2 + nm + 1) } - 1 , $$ where $ M ( m , n ) $ is equal to $$ \begin{array}{llll} \mathcal{C}_7 ( 2 + n^2 , t, u) & + \mathcal{A}_7( - 2 m x_1 x_2 , t, u) & + \mathcal{A}_7( ( m^2 + 1 ) x_1^2 , t, u) & + \mathcal{A}_7( - x_4 2^{x_1 + 1} , t, u) \\ + \mathcal{A}_7( - 2 x_1 , t, u) & + \mathcal{A}_7( - 2 ( m + 1 ) x_1 x_7 , t, u) & + \mathcal{A}_7( x_2^2 , t, u) & + \mathcal{A}_7( - x_5 2^{x_1 + 1} , t, u) \\ + \mathcal{A}_7( 2 x_5 , t, u) & + \mathcal{A}_7( 2 x_4 x_5 , t, u) & + \mathcal{A}_7( 3 x_4^2 , t, u) & + \mathcal{A}_7( - x_7 2^{x_1 + 1} , t, u) \\ + \mathcal{A}_7( 2 ( 1 - n ) x_4 , t, u) & + \mathcal{A}_7( 2 x_4 x_6 , t, u) & + \mathcal{A}_7( x_5^2 , t, u) & + \mathcal{A}_7( - x_4 2^{x_2 + 1} , t, u) \\ + \mathcal{A}_7( - 2 n x_6 , t, u) & + \mathcal{A}_7( 2 x_4 x_7 , t, u) & + \mathcal{A}_7( x_6^2 , t, u) & + \mathcal{A}_7( x_3 x_4 2^{x_1 + 1} , t, u) \\ + \mathcal{A}_7( 2 ( m + 2 ) x_7 , t, u) & + \mathcal{A}_7( 2 x_5 x_7 , t, u) & + \mathcal{A}_7( ( m^2 + 2 m + 2 ) x_7^2 , t, u) & + \mathcal{A}_7( x_3 x_7 2^{x_2 + 1} , t, u) \\ + \mathcal{A}_7( - 2 x_3 x_4 x_7 , t, u) & + \mathcal{A}_7( - 2^{x_1 + 1} , t, u) & + \mathcal{A}_7( 2^{2 x_1} , t, u) & + \mathcal{A}_7( - x_3^2 x_7 2^{x_1 + 1} , t, u) \\ + \mathcal{A}_7( x_3^2 x_7^2 , t, u) & + \mathcal{A}_7( x_3^2 2^{2 x_1} , t, u) & + \mathcal{A}_7( 2^{2 x_2} , t, u) & + \mathcal{A}_7( - x_3 2^{x_1 + x_2 + 1} , t, u) . \end{array} $$
\end{lemma}

\section{Appendix: Source code} \label{sectionsourcecode}

This appendix provides SageMath code for implementing and verifying the arithmetic terms developed throughout this paper. The code demonstrates the construction of $\chi(n)$ (the largest square divisor), $\omega(n)$ (the number of distinct prime divisors), and the main factoring terms $T(n)$ and $U(n)$ from \cref{sectionfactoringterm}.

The implementation follows a step-wise approach, building up the components systematically:
\begin{enumerate}
\item Helper functions implementing the hypercube method (generalized geometric series $G_r(q,t)$, arithmetic term constructors $\mathcal{A}$ and $\mathcal{C}$)
\item Arithmetic terms for $\chi(n)$ (\cref{proof:chi}) and $\omega(n)$ (\cref{proof:omegasolutions})
\item Supporting terms for binomial coefficients, factorial, and greatest common divisor
\item The complete factoring terms $T(n)$ (\cref{proof:divisor1}) and $U(n)$ (\cref{proof:divisor})
\item Verification tests demonstrating that both terms successfully produce proper divisors for composite integers
\end{enumerate}

The arithmetic term formulas presented in this paper are computationally infeasible for large $n$ in practice, so the purpose of this code is merely to provide concrete verification that the constructions are mathematically sound. To make verifications practical, the code uses SageMath's built-in functions where appropriate (e.g. for factorial and GCD operations), as the full arithmetic term implementations become intractable even for moderate values. However, the arithmetic term versions of these functions (see \cref{sectionpreliminaries}) are also included and can be easily substituted by modifying the code.

{\fontsize{8}{9}\selectfont
\begin{verbatim}
    from sage.all import *

    def print_values(values, title=''):
        """Print values in a comma-separated format."""
        if title != '': print(f'{title}:')
        for v in values: print(f'{v}', end=',')
        print('')
    
    def nu2(a):
        """2-adic valuation of a."""
        if a == 0: return +Infinity
        return Integer(a).valuation(2)
    
    def HW(a):
        """Hamming weight (number of 1s in binary representation)."""
        return bin(Integer(a)).count('1')
    
    def G0(q, t):
        """G_0(q,t): geometric series."""
        if q == 1: return t
        return (q**t - 1) / (q - 1)
    
    def G1(q, t):
        """G_1(q,t): weighted geometric series with j."""
        if q == 1: return t*(t-1)/2
        t1 = t - 1
        return q * (t1 * q**t - t * q**t1 + 1) / (q - 1)**2
    
    def G2(q, t):
        """G_2(q,t): weighted geometric series with j^2."""
        if q == 1: return t*(t-1)*(2*t-1)/6
        t1 = t - 1
        return q * (
            t1**2 * q**(t1 + 2)
            - (t1**2 * 2 + t1 * 2 - 1) * q**(t1 + 1)
            + (t1 + 1)**2 * q**t1 - q - 1
        ) / (q - 1)**3
    
    def G4(q, t):
        """G_4(q,t): weighted geometric series with j^4."""
        if q == 1:
            s = sum(j**4 for j in range(t))
            return s
        t1 = t - 1
        return q * (
            t1**4 * q**(t1 + 4)
            + (t1**4 * (-4) - 12 * t1**3 - 6 * t1**2 + t1 * 12 + 11) * q**(t1 + 1)
            + (t1**4 * 6 + 12 * t1**3 - 6 * t1**2 - t1 * 12 + 11) * q**(t1 + 2)
            + (t1**4 * (-4) - 4 * t1**3 + 6 * t1**2 - t1 * 4 + 1) * q**(t1 + 3)
            + (t1 + 1)**4 * q**t1 - q**3 - q**2 * 11 - q * 11 - 1
        ) / (q - 1)**5
    
    def G(r, q, t):
        """Generalized geometric series G_r(q,t) = sum_{j=0}^{t-1} j^r * q^j."""
        if r == 0: return G0(q, t)
        elif r == 1: return G1(q, t)
        elif r == 2: return G2(q, t)
        elif r == 4: return G4(q, t)
        else:
            # Fallback for other powers
            g = QQ(0)
            qj = QQ(1)
            for j in range(t):
                g += qj * j**r
                qj *= q
            return g
    
    def C(a, k, t, u):
        """Free term C_k(a, t, u)."""
        return ((2**u - a + 1) * (2**(u * 2 * t**k) - 1)) / (2**u + 1)
    
    def A(a, P, B, V, k, t, u):
        """Arithmetic term A_k for monomial from equation.
        
        Parameters:
        - a: coefficient
        - P: list of powers [r_1, ..., r_k]
        - B: list of bases (usually [2,2,...])
        - V: list of exponential powers for bases
        - k: dimension
        - t, u: parameters
        """
        result = -(2**u - 1) * a
        for i in range(k):
            q = B[i]**V[i] * 2**(u * t**i * 2)
            g = G(P[i], q, t)
            result *= g
        return result
    
    def M_chi(n):
        """Construct M(f, n, n, n+4) for chi function.
        
        Based on f(x,y) = (x^2 - ny)^2 = x^4 - 2n*x^2*y + n^2*y^2
        """
        k = 2
        t = n
        u = n + 4
        B = [2, 2]  # Bases
        V = [0, 0]  # Exponential powers for bases
        
        return (C(0, k, t, u)
                + A(1, [4, 0], B, V, k, t, u)           # x^4
                + A(-2*n, [2, 1], B, V, k, t, u)        # -2n*x^2*y
                + A(n**2, [0, 2], B, V, k, t, u))       # n^2*y^2
    
    def Chi(n):
        """Chi(n): largest s such that s^2 divides n."""
        if n == 0: return 0
        M = M_chi(n)
        return Integer(HW(M) / (n + 4) - n**2)
    
    def M_omega(n):
        """Construct M(g, 4n, 4n, 4n+4) for omega function.
        
        Based on g(x,y) = (x^2 - ny - 1)^2 = x^4 - 2x^2 - 2n*x^2*y + n^2*y^2 + 2ny + 1
        """
        k = 2
        t = n
        u = n + 4
        B = [2, 2]
        V = [0, 0]
        
        return (C(1, k, t, u)
                + A(1, [4, 0], B, V, k, t, u)           # x^4
                + A(-2, [2, 0], B, V, k, t, u)          # -2x^2
                + A(-2*n, [2, 1], B, V, k, t, u)        # -2n*x^2*y
                + A(n**2, [0, 2], B, V, k, t, u)        # n^2*y^2
                + A(2*n, [0, 1], B, V, k, t, u))        # 2ny
    
    def Omega(n):
        """Omega(n): number of distinct prime divisors of n."""
        if n == 0 or n == 1: return 0
        M = M_omega(4*n)
        N = HW(M) / (4*n + 4) - (4*n)**2
        return nu2(Integer(N)) - 1
    
    def Binomial_term(a, b):
        """Binomial coefficient using arithmetic term formula."""
        if b > a or b < 0: return 0
        numerator = 2**(2*(a+2)*((a+1)**2 + b + 1))
        denominator = 2**(2*(a+2)**2) - 2**(2*(a+2)) - 1
        return Integer(floor(numerator / denominator) % 2**(2*(a+2)))
    
    def Factorial_term(n):
        """Factorial using arithmetic term formula."""
        if n == 0: return 1
        binom = Binomial_term(8*n**2, n)
        return Integer(floor(8**(n**3) / binom))
    
    def Gcd_term(a, b):
        """GCD using arithmetic term formula."""
        if a == 0: return b
        if b == 0: return a
        numerator = 2**(a*b*(a*b + a + b))
        denominator = (2**(a**2 * b) - 1) * (2**(a*b**2) - 1)
        return Integer((floor(numerator / denominator) % 2**(a*b)) - 1)
    
    def FloorRoot_simple(m, n):
        """Simplified m-th root using native computation.
        
        Note: The full arithmetic term implementation is extremely
        complex. For verification purposes, we use Sage's built-in computation.
        """
        if n == 0: return 0
        if m == 0: return 1
        return Integer(floor(n**(1/m)))
    
    def T(n):
        """Main factoring term T(n).
        
        T(n) = gcd(n/chi(n), floor(omega(n)-th root of n)!)
        
        For composite n, returns a non-trivial divisor.
        """
        if n <= 1: return n
        
        chi_n = Chi(n)
        if chi_n == 0: chi_n = 1  # Avoid division by zero
        
        omega_n = Omega(n)
        if omega_n == 0: omega_n = 1  # Handles prime case
        
        # Compute floor(omega(n)-th root of n)
        root_val = FloorRoot_simple(omega_n, n)
        fact_val = factorial(root_val)  # Use built-in factorial for practical computation
        
        # Compute GCD
        arg1 = Integer(n / chi_n)
        result = gcd(arg1, fact_val)  # Use built-in gcd for practical computation
        
        return result
    
    def U(n):
        """Alternative factoring term U(n).
        
        U(n) = (2 - chi(n)) * gcd(n, floor(omega(n)-th root of n)!)
               + (1 - (2 - chi(n))) * chi(n)
               
        Uses truncated difference (monus operation).
        """
        if n <= 1: return n
        
        chi_n = Chi(n)
        omega_n = Omega(n)
        if omega_n == 0: omega_n = 1
        
        root_val = FloorRoot_simple(omega_n, n)
        fact_val = factorial(root_val)  # Use built-in factorial for practical computation
        gcd_val = gcd(n, fact_val)
        
        # Monus operation: a - b if a >= b else 0
        monus1 = max(2 - chi_n, 0)
        monus2 = max(1 - monus1, 0)
        
        result = monus1 * gcd_val + monus2 * chi_n
        return Integer(result)
    
    print("="*70)
    print("Testing Chi(n): Largest square divisor")
    print("="*70)
    n_values = list(range(1,32))
    print_values(n_values, 'n')
    print_values([Chi(n) for n in n_values], 'Chi(n)')
    print()
    
    print("="*70)
    print("Testing Omega(n): Number of distinct prime divisors")
    print("="*70)
    n_values = list(range(1, 32))
    print_values(n_values, 'n')
    print_values([Omega(n) for n in n_values], 'Omega(n)')
    print()
    
    print("="*70)
    print("Testing T(n): Main factoring term for composite numbers")
    print("="*70)
    composite_values = [Integer(n) for n in range(2, 50) if not Integer(n).is_prime()]
    for n in composite_values:
        t_n = T(n)
        # Verify it's a proper divisor
        is_proper = 1 < t_n < n
        factors = f"{t_n} x {n//t_n}" if t_n != 0 and n % t_n == 0 else "error"
        status = "Pass" if is_proper else "Fail"
        print(f"n={n:3d}: T(n)={t_n:3d}, factors: {factors:12s} {status}")
    print()
    
    print("="*70)
    print("Testing U(n): Alternative factoring term")
    print("="*70)
    for n in composite_values:
        u_n = U(n)
        is_proper = 1 < u_n < n
        factors = f"{u_n} x {n//u_n}" if u_n != 0 and n % u_n == 0 else "error"
        status = "Pass" if is_proper else "Fail"
        print(f"n={n:3d}: U(n)={u_n:3d}, factors: {factors:12s} {status}")
    print()
\end{verbatim}
}

%%%%%%%%%%%%%%%%%%%%%%%%%%%%%%%%%%%%%%%%%%%%%%%%%%%%%%%%%%%%%%%%%%%%

%\begingroup
%\raggedright
%\bibliographystyle{plainnat}
%\bibliography{main}
%\endgroup

\end{document}